\documentclass[creativecommons,sharealike]{eptcs}
\providecommand{\event}{VSS 2025} % Name of the event you are submitting to
\usepackage[T1]{fontenc}
\usepackage[utf8]{inputenc}

\usepackage{iftex}

\ifpdf
  \usepackage{underscore}         % Only needed if you use pdflatex.
  \usepackage[T1]{fontenc}        % Recommended with pdflatex
\else
  \usepackage{breakurl}           % Not needed if you use pdflatex only.
\fi

\DeclareUnicodeCharacter{03B1}{\texttt{$\alpha$}} % For \alpha in Coq listings

\usepackage{enumitem}
\usepackage{amsmath,amssymb,amsthm}
\usepackage{tablefootnote}
\newtheorem{theorem}{Theorem}[section]
\newtheorem{definition}{Definition}[section]

\DeclareMathOperator{\rnd}{\rho}
\DeclareMathOperator{\op}{\mathit{op}}

\newcommand{\apb}{\mathbin{;}}
\DeclareMathOperator{\ap}{ap}
\DeclareMathOperator{\rp}{rp}
\DeclareMathOperator{\rel}{\mathit{rel}}
\DeclareMathOperator{\abs}{\mathit{abs}}
\DeclareMathOperator{\F}{\mathbb{F}}
\DeclareMathOperator{\R}{\mathbb{R}}
\DeclareMathOperator{\err}{\textnormal{err}}
\newcommand{\repo}{https://github.com/Athena-Types/mechanized-precision}
\newcommand{\slug}{tree/main/theories}

\newtheorem{corollary}{Corollary}[theorem]

\newcommand{\coqref}[1]{(\href{#1}{$\lozenge$})}
\newcommand{\coqurl}[1]{\href{#1}{$\lozenge$}}
\newcommand{\coqfile}[1]{(\href{\repo/\slug/#1}{$\lozenge$})}
\newcommand{\coqfilen}[1]{\href{\repo/\slug/#1}{$\lozenge$}}

\newcommand{\coqtag}[1]{% \coqtag{url}
  \refstepcounter{equation}\tag{\theequation, \href{#1}{$\lozenge$}}%
}
\newcommand{\coqtagf}[1]{% \coqtagf{file_name.v\#L___}
  \refstepcounter{equation}\tag{\theequation, \href{\repo/\slug/#1}{$\lozenge$}}%
}

\makeatletter
\newcommand{\owntag}[2][\relax]{% \owntag[short label]{tag}
  \ifx#1\relax\relax\def\owntag@name{#2}\else\def\owntag@name{#1}\fi% base label
  \refstepcounter{equation}\tag{\theequation, #2}%
  \expandafter\ltx@label\expandafter{eq:\owntag@name}%
  \edef\@currentlabel{\theequation, #2}\expandafter\ltx@label\expandafter{Eq:\owntag@name}%
  \def\@currentlabel{#2}\expandafter\ltx@label\expandafter{tag:\owntag@name}%
}
\makeatother

\usepackage{newtxmath} % Default math font with greek looks weird to me using EPTCS. eulervm looks really nice but is perhaps a bit too different from the standard. newpxmath is too... italic.

\usepackage{cleveref}

\title{Mechanizing Olver's Error Arithmetic}
\author{
Max Fan
\institute{Cornell University \\
Ithaca, New York, USA}
\email{mxf@cs.cornell.edu}
\and
Ariel E. Kellison
\institute{Sandia National Laboratories\\
Livermore, CA, USA}
\email{aekelli@sandia.gov}
\and
Samuel D. Pollard
\institute{Sandia National Laboratories\\
Livermore, CA, USA}
\email{spolla@sandia.gov}
}
\def\titlerunning{Mechanizing Olver's Error Arithmetic}
\def\authorrunning{M. Fan, A. E. Kellison, \& S. D. Pollard}
\begin{document}
\maketitle

\begin{abstract}
We mechanize the fundamental properties of a rounding error model for floating-point arithmetic based on \emph{relative precision}, a measure of error proposed as a substitute for relative error in rounding error analysis. A key property of relative precision is that it forms a true metric, providing a well-defined measure of distance between exact results and their floating-point approximations while offering a structured approach to propagating error bounds through sequences of computations. Our mechanization formalizes this property, establishes rules for converting between relative precision and relative error, and shows that the rounding error model based on relative precision slightly overapproximates the standard rounding error model. Finally, we demonstrate, with a simple example of the inner product of two vectors, that this alternative model facilitates a tractable approach to developing certified rounding error bounds.
\end{abstract}

\section{Introduction}
Many scientific and engineering domains rely on algorithms designed to solve problems in continuous mathematics. While these algorithms are typically specified using real numbers, their implementations use floating-point arithmetic. As a result, these implementations deviate from their real-number specifications due to rounding and approximation errors. These deviations make it challenging to ensure the correctness of programs that rely on floating-point arithmetic, motivating the development of numerous tools for analyzing and verifying floating-point computations~\cite{appel24,solovyev18,Darulova:2018:Daisy,colibri21,dedinechin:2011:gappa}. In high-consequence settings, practitioners often turn to formal verification techniques, including proof assistants~\cite{pollard:2023:twofv}, to certify rounding error estimates and ensure rigorous guarantees of correctness.

A fundamental decision in both manual rounding error analyses and formal verification is the choice of a \emph{rounding error model}: how does the floating-point result of a basic arithmetic operation relate to its exact real-number counterpart? This decision is crucial because it directly impacts the \emph{soundness}, \emph{precision}, and \emph{tractability} of the analysis. A model that underapproximates floating-point behavior may lead to \emph{unsound guarantees}, while an overly conservative model may introduce \emph{pessimistic} error bounds that limit its practical usefulness. More precise models improve accuracy but often increase the complexity of verification, making proofs harder to automate. Choosing the right model is therefore a key trade-off in certified error estimation.  

In this work, we provide a mechanized formalization of an alternative rounding error model based on \emph{relative precision}, first proposed by Olver~\cite{olver78} as a substitute for relative error in rounding error analysis. Unlike relative error, relative precision defines a true metric, providing a well-defined measure of distance between floating-point approximations and exact results. By construction, relative precision naturally supports an alternative rounding error model that succinctly and soundly captures higher-order error terms. We mechanize the fundamental properties of this model in the Rocq\footnote{Rocq was renamed from Coq in 2025.} proof assistant~\cite{coq820} and establish its relationship to the standard rounding error model for elementary floating-point operations formalized in Flocq~\cite{boldo11}.

Mechanizing previously established pen-and-paper results strengthens the foundations for automated and mechanized sound analyses, leading to more reliable tools for reasoning about floating-point error. Formalization in a proof assistant not only increases confidence in correctness but also enhances extensibility, enabling future extensions and integrations with automated verification frameworks. For example, by mechanizing Olver's work, we discovered that two fundamental properties of relative precision from the original paper did not adequately handle signs and were thus ill-defined over negative values.

More concretely, this work contributes:

\begin{itemize}
    \item A detailed account of Olver's original error arithmetic (\Cref{sec:formal}), with the results proved in Rocq.
    % \item Two modified properties for relative precision that are ill-defined in Olver's original paper and are now well-defined (\Cref{sec:rp}).
    \item Corrections to make two ill-defined properties of relative precision in Olver's original error arithmetic well-defined and true (\Cref{sec:rp}).
    \item A demonstration of the simplicity of the proposed error analysis using an inner product example in Rocq (\Cref{sec:example}).
\end{itemize}

We release our code at \url{\repo} under an open-source license and provide hyperlinks to the Rocq code using the $\lozenge$ symbol.
  
\section{Overview}
A rounding error analysis aims to provide an \emph{a priori} estimate of how rounding errors impact an algorithm. The typical approach models the rounding error from a single floating-point operation using a sound overapproximation of the true error, and then combines the error bounds from each individual operation to approximate the total error accumulated over a sequence of operations. We provide the necessary definitions below, but we refer the interested reader to~\cite[Ch.~2]{Muller18} for a more complete discussion on floating point, and rounding, overflow, underflow, and exceptional values such as NaN and $\infty$.

\paragraph{The Standard Rounding Error Model}
A widely used model for rounding error analysis is the \emph{standard model} of floating-point arithmetic \cite{higham02}, which describes how floating-point approximations relate to exact computations under the assumption that overflow, underflow, and exceptional values do not occur. The central theorem of this model, for a given rounding function $\rnd : \mathbb{R} \rightarrow \mathbb{F}$, is as follows, where $\mathbb{R}$ and $\mathbb{F}$ are the sets of all real and floating-point numbers, respectively.
\begin{theorem}\label{thm:main_round}
  Given a real number $x \in \R$, and assuming no underflow or overflow occurs, 
  then there exists a $\delta \in \mathbb{R}$ such that \cite[Theorem 2.2]{higham02}: 
\begin{equation}\label{eq:main_round}
    \rho(x) = x (1+\delta) \qquad |\delta| \le u,
\end{equation}
\end{theorem}

\begin{table}
\centering
\caption{The behavior of common rounding functions and their corresponding unit roundoff values ($u$) for binary formats. $\mathbb{F}$ is the set of all floating-point numbers for a given format with precision $p$.}
\label{tab:rnd_modes}
\begin{tabular}{ l l l l }
 \hline
\textbf{Rounding Function ($\rho$)} & \textbf{Behavior} & \textbf{Notation} & \textbf{Unit Roundoff ($u$)} \\
\hline
 Round towards $+\infty$   & $\min\{b \in \F \mid b \ge a \}$  & $\rho_{RU}(a)$ & $2^{1-p}$   \\
 Round towards $-\infty$   & $\max\{b \in \F \mid b \le a \}$  &   $\rho_{RD}(a)$ & $2^{1-p}$ \\
 Round towards $0$   &   $\rho_{RU}(a)$ if $a < 0 $, otherwise  $\rho_{RD}(a)$  & $\rho_{RZ}(a)$
        & $2^{1-p} $ \\
 Round towards nearest & $\{b \in \F \mid \forall c
	\in \F,  |a - b| \le |a-c| \}
         $   & $\rho_{RN}(a)$ & $2^{-p}  $    \\
 \hline
\end{tabular}
\end{table}

The constant $u$ in \Cref{thm:main_round} is the unit roundoff, which depends on the floating-point format (e.g., IEEE 754 single precision or double precision) and the rounding function $\rnd : \mathbb{R} \rightarrow \mathbb{F}$; the behavior of common rounding functions and their corresponding unit roundoff values for binary formats are given in \Cref{tab:rnd_modes}. For IEEE 754 single and double precision with round-to-nearest (with arbitrary tie breaking), the unit roundoff values are $u = 2^{-24}$ and $u = 2^{-53}$, respectively~\cite{IEEE754:19}. 

The standard model for floating-point arithmetic follows from \Cref{thm:main_round}: for floating-point numbers $x$ and $y$ in a format $\mathbb{F} \subset \mathbb{R}$ (again, ignoring exceptional values like NaN or $\infty$), for real numbers $a$ and $b$, the standard model expresses the floating-point approximation $\rnd(a ~\op~ b) \in \mathbb{F}$ of $(a \;\op\; b) \in \mathbb{R}$ as
\begin{equation}\label{eq:stdmodel}
    \rnd(a \;\op\; b) = (a \;\op\; b) (1 + \delta), \qquad |\delta| \le u, \qquad \op \in \{+,-,\times,\div\}.
\end{equation}

\paragraph{Absolute and Relative Error}
The standard rounding error model given in \Cref{eq:stdmodel} provides a structured way to approximate floating-point errors. Within this framework, two fundamental measures of error are widely used: absolute error and relative error. As functions, relative error (err$_{\rel}$) and absolute error (err$_{\abs}$) are defined over the real numbers $\mathbb{R}$, where the set of representable floating-point numbers in a specified format $\mathbb{F}$ form a discrete subset:  

\begin{align}
    \err_{\rel} (x,y) \triangleq \left| \frac{x-y}{x} \right|
        \label{eq:rel_error}\\
    \err_{\abs} (x,y) \triangleq  |x-y| \label{eq:abs_error}
\end{align}

Absolute error and relative error each come with distinct trade-offs. Absolute error disproportionately weights discrepancies between larger values more than smaller ones. Relative error, on the other hand, behaves poorly on small values, where it can become unbounded, and also fails to satisfy the axioms of a metric on $\mathbb{R}$. Specifically, relative error does not satisfy the property of symmetry, as $\err_{\rel} (x,y) \neq \err_{\rel} (y,x)$, and also violates the triangle inequality, making it unsuitable for structured reasoning about the propagation of rounding errors. A measure of error that forms a true metric provides a natural way to reason about the \emph{composition} of errors, ensuring a tractable method for propagating error bounds through sequences of computations.  

\paragraph{Metrics and Rounding Error Bounds}
For instance, consider two real numbers $a, b \in \mathbb{R}$. The exact sum of $a$ and $b$ is given by  
\[
s = a + b,
\]  
while the floating-point computation, using \Cref{eq:stdmodel}, produces  
\[
\tilde{s} = (a + b)(1+\delta) = \tilde{a} + \tilde{b}; \quad |\delta| \le u,
\] 
where $\tilde{a}=a(1+\delta)$ and $\tilde{b}=b(1+\delta)$. If $d_\mathbb{R}: \mathbb{R} \times \mathbb{R} \to \mathbb{R}^{\geq 0}$ is a true metric that satisfies the triangle inequality, then the total rounding error can be bounded as follows:  
\begin{equation}\label{eq:ex1}
d_\mathbb{R}(s, \tilde{s}) \leq d_\mathbb{R}(a, \tilde{a}) + d_\mathbb{R}(b, \tilde{b}).
\end{equation}  
Here, the first term $d_\mathbb{R}(a, \tilde{a})$ measures the error in approximating $a$ by $\tilde{a}$, and similarly, $d_\mathbb{R}(b, \tilde{b})$ quantifies the error in approximating $b$ by $\tilde{b}$. The triangle inequality ensures that these local errors can be systematically combined to bound the total error in the sum.  

This property allows for a structured approach to error analysis: if individual rounding errors can be bounded, then their composition can be rigorously bounded using the metric structure. More generally, the same principle applies to sequences of floating-point computations beyond addition, enabling a tractable method for bounding accumulated errors in numerical algorithms.

\paragraph{Relative Precision} A metric that closely approximates relative error, known as \emph{relative precision}, was first introduced by Olver~\cite{olver78}. To simplify rounding error analysis, Olver proposed the notation 
\begin{equation}\label{eq:rp1}
    a \sim \tilde{a}; \quad \text{rp}(\alpha)
\end{equation}
to indicate that 
\begin{equation}\label{eq:rp2}
   \tilde{a} / a = e^\delta, \ \ \text{ where } |\delta| \le \alpha,
\end{equation}
and referred to $\tilde{a}$ as an approximation of $a$ with \emph{relative precision} $\alpha$. Notably, \Cref{eq:rp2} only holds whenever $a$ and $\tilde{a}$ are both nonzero and have the same sign. This notation naturally leads to the following definition of a \emph{relative precision metric} on nonzero numbers of the same sign: 

\begin{equation}
    \err_{\rp}(\tilde{x},x) \triangleq |\ln(\tilde{x}/x)|  \ \ \text{ where }\ x,\tilde{x} \neq 0 \ \text{ and } \ \tilde{x}/x > 0. 
\end{equation}

We next describe our Rocq formalization of the relative precision metric, a rounding error model based on \Cref{eq:rp2}, and the key properties of both the metric and the model. 

\section{Formalization}\label{sec:formal}

This section outlines the key theorems from our Rocq formalization of an alternative rounding error model for floating-point arithmetic based on the relative precision metric.

\subsection{An Alternative Model for Floating-Point Arithmetic}\label{ssec:alt_model}
In the spirit of \Cref{eq:rp2}, our mechanization establishes the following alternative to \Cref{thm:main_round}. 
\begin{theorem}\label{thm:rp_model}
Given a real number $x \in \R$, and a rounding function $\rho \in \{\rho_{RU},\rho_{RD},\rho_{RZ},\rho_{RN}\}$ and assuming no underflow or overflow occurs, then there exists a $\delta \in \mathbb{R}$ such that
%\refstepcounter{equation}
\begin{equation} 
   \rho(x) = x e^\delta, \ \ \text{ where } |\delta| \le \frac{u}{1-u},
    \coqtagf{error_model.v\#Ll67}
\end{equation}
where $0 < u \ll 1$ is the unit roundoff for $\rho$.
\end{theorem}
Notably, the rounding error bound in \Cref{thm:rp_model} is a slight overapproximation to the bound in \Cref{thm:main_round}. This becomes clearer when considering the Taylor series expansion of $e^\delta$:
\begin{equation} 
   \rho(x) = x e^\delta = x (1 + \delta + O(\delta^2))
\end{equation}
Additionally, note that $u \le u/(1-u)$. Furthermore, while Olver's original notation in \Cref{eq:rp2} applies only when $a$ and $\tilde{a}$ are both nonzero and have the same sign, \Cref{thm:rp_model} remains valid for any real number $x \in \R$, provided no underflow or overflow occur.   

From \Cref{thm:rp_model}, we can derive an alternative model for floating-point arithmetic: for floating-point numbers $x, y \in \mathbb{F}$, the alternative model expressed the floating-point approximation $\rho(x ~\op~ y) \in \F$ of $(x ~\op~ y) \in \R$ as 
\begin{equation}
    \rho(x ~\op~ y)  = (x ~\op~ y)e^\delta, \quad |\delta| \le \frac{u}{1-u}, \quad \op \in \{+,-,\times,\div\}.
\end{equation}

The key advantage of this model is that it represents rounding errors in a way that naturally aligns with the relative precision metric, enabling a true metric-based characterization of the accuracy of a floating-point operation relative to its exact result. In the following section, we outline the key definitions and properties of relative precision, as first introduced by Olver~\cite{olver78}, which we have formalized in Rocq. We additionally discuss two properties that we discovered to be ill-defined in the course of mechanization, and propose fixes to make the properties well-defined. In \Cref{sec:example}, we demonstrate how relative precision provides a convenient and mathematically sound framework for analyzing floating-point computations. 

\subsection{Relative Precision}\label{sec:rp}
\begin{definition}[Relative Precision \coqfile{relative_prec.v\#L8}] \label{def:rp}
Let $a, \tilde{a}, \alpha \in \mathbb{R}$. We say that $\tilde{a}$ is an approximation to $a$ of 
\emph{relative precision} $\alpha$ \emph{iff}  $\alpha \ge 0$, $a$ and $\tilde{a}$ are nonzero and have the 
same sign, and
\[
    \err_{\rp}(\tilde{a},a) = |\ln(\tilde{a}/a)| \le \alpha.
\] 
\end{definition}

\noindent Recall from \Cref{eq:rp1} $a \sim \tilde{a} \;\apb\; \rp(\alpha)$ denotes that $\tilde{a}$ is an approximation to $a$ of \emph{relative precision} $\alpha$. To convert from relative precision to relative error, we prove the following theorem:

\begin{theorem}\label{thm:rp_re_conversion}
Let $a$, $\tilde{a}$, $\alpha \in \mathbb{R}$ such that $a \sim \tilde{a} \;\apb\; \rp(\alpha)$. Then,
\begin{equation}
\err_{\rel} (a, \tilde{a}) \le |e^\alpha - 1|.
\coqtagf{relative_error.v\#L21}
\end{equation}
\end{theorem}

Given the above definition of relative precision and \Cref{thm:rp_model}, we formalized the following theorem in Rocq, which characterizes the relative error of basic floating-point operations in terms of relative precision.

\begin{theorem}\label{thm:rp_ops}
Let $a,b \in \mathbb{F}$ be two floating-point numbers such that $a + b \ne 0$. 
Assuming overflow and underflow do not occur, then
\begin{equation}
\rho(a \op b) \;\apb\; \rp \left(\frac{u}{1-u}\right), \quad \op \in \{+,-,\times,\div\}
\coqtagf{relative_prec_ops.v\#L42}
\end{equation}
for some $0 < u \ll 1$.
\end{theorem}

In \Cref{sec:example}, we use this result to establish overall bounds on a floating-point inner product. Deriving these bounds relies on several desirable properties of relative precision, which we have proved in our mechanization and list below. For each, suppose $a \sim \tilde{a} \;\apb\; \rp(\alpha)$.

\begin{enumerate}[label=RP\Roman*.]
\item Symmetry \coqfile{relative_prec.v\#L92}: $\tilde{a} \sim a \;\apb\; \rp(\alpha)$
    \label{prop:i}
\item Inclusion \coqfile{relative_prec.v\#L127}:
    for all $\delta \ge \alpha$, $a \sim \tilde{a} \;\apb\; \rp(\delta)$.
    \label{prop:ii} 
\item Constant scaling \coqfile{relative_prec.v\#L136}:
    If $k \ne 0$, then $ka \sim k\tilde{a} \;\apb\; \rp(\alpha)$.
    \label{prop:iii} 
\item Exponentiation, modified to take absolute values --- originally ill-defined \coqfile{relative_prec.v\#L149}: \\
    If $k \in \mathbb{R}$, then $|a|^k \sim |\tilde{a}|^k \;\apb\; \rp(|k|\alpha)$.
    \label{prop:iv}
\item Composition \coqfile{relative_prec.v\#L168}:
    If $b \sim \tilde{b} \;\apb\;\rp(\beta)$, then $a b \sim \tilde{a} \tilde{b} \;\apb\; \rp(\alpha + \beta)$.
    \label{prop:v}
\item Triangle inequality \coqfile{relative_prec.v\#L201}:
    If $\tilde{a} \sim \tilde{\tilde{a}} \;\apb\; \rp(\delta)$, then $a \sim \tilde{\tilde{a}} \;\apb\; \rp(\alpha + \delta)$.
    \label{prop:vi}
\end{enumerate}

Through formalization, we also uncovered issues in Olver's property statements. In particular, observe that property \ref{prop:iv} is ill-defined over the reals for negative $a$ and small $k$: $a^k \sim \tilde{a}^k \;\apb\; \rp(|k|\alpha)$. To address this, we modify the property to take absolute values prior to exponentiation. 

In contrast, for relative error (\Cref{eq:rel_error}), properties \ref{prop:i} and \ref{prop:vi} do not hold. 
For \ref{prop:i}, consider $a = 1$ and $\tilde{a} = 1.1$. 
Then, $\err_{\rel}(a, \tilde{a}) = \frac{1}{10}$.
But, $\err_{\rel}(\tilde{a}, a) = \frac{1}{11}$.
For \ref{prop:vi}, consider $a = 1$, $\tilde{a} = 1.5$, and $\tilde{\tilde{a}} = 2$. Then, $\err_{\rel}(a, \tilde{a}) = \frac{1}{2}$ and $\err_{\rel}(\tilde{a}, \tilde{\tilde{a}}) = \frac{1}{3}$ but $\text{err}_{\rel}(a, \tilde{\tilde{a}}) = 1 > \frac{1}{2} + \frac{1}{3}$.

As a result, relative precision satisfies the axioms of a metric on $\mathbb{R}^{>0}$, whereas relative error does not.
From the above properties, the following theorems describe how relative precision propagates due to addition.

\begin{theorem} \label{thm:rp_add}
Let $a \sim \tilde {a} \;\apb\; \rp(\alpha)$ and let $b \sim \tilde {b} \;\apb\; \rp(\beta)$ and assume that $a$ and $b$ have the same sign and overflow and underflow do not occur. Then
\begin{equation}
    a + b \sim \tilde a + \tilde b \;\apb\; \rp \left( \ln
        \left( \frac{\tilde{a} e ^\alpha + \tilde{b}e^\beta}
                    {\tilde{a} + \tilde{b}} \right) \right).
    \coqtagf{relative_prec.v\#L404}
\end{equation}
\end{theorem}

A useful corollary to \Cref{thm:rp_add} is the following, which leverages the inclusion property (Property II) for relative precision to derive a somewhat looser but more practical bound on how relative precision propagates under addition.

\begin{corollary} \label{cor:rp_add}
Let $a \sim \tilde {a} \;\apb\; \rp(\alpha)$ and let $b \sim \tilde {b} \;\apb\; \rp(\beta)$ and assume that $a$ and $b$ have the same sign. Then
\begin{equation}
    a + b \sim \tilde a + \tilde b \;\apb\; \rp (\max(\alpha,\beta)).
    \coqtagf{relative_prec.v\#L429}
\end{equation}
\end{corollary}

We next describe \emph{absolute precision}, another convenient formalism for describing floating-point error.

\subsection{Absolute Precision}

We begin with the definition of absolute precision:
\begin{definition}[Absolute Precision \coqfile{absolute_prec.v\#L8}]
Let $x, \tilde{x}$ be real numbers. If
\begin{equation}
    |x - \tilde{x}| \le \alpha
\end{equation}
for some $\alpha \ge 0$, then $\tilde{x}$ is 
an approximation to $x$ of \emph{absolute precision} $\alpha$.
\end{definition}

Following Olver, for any $a, \tilde{a} \in \mathbb{R}$, we write 
\begin{equation}
a \sim \tilde{a} \;\apb\; \ap(\alpha)
\end{equation}
to denote that $\tilde{a}$ approximates $a$ with absolute precision $\alpha$. Clearly, if 
$a \sim \tilde{a} \;\apb\; \ap(\alpha)$ then the \emph{absolute error} (\Cref{eq:abs_error}) of the approximation $\tilde{a}$ with respect to $a$ is also bounded by $\alpha$. 

Absolute precision enjoys several nice properties which we have proved in our mechanization. For each, suppose $a \sim \tilde{a} \;\apb\; \ap(\alpha)$.

\begin{enumerate}[label=AP\Roman*.]
\item Symmetry \coqfile{absolute_prec.v\#L21}: $\tilde{a} \sim a \;\apb\; \ap(\alpha)$
    \label{prop:ai}
\item Inclusion \coqfile{absolute_prec.v\#L25}:
    for all $\delta \ge \alpha$, $a \sim \tilde{a} \;\apb\; \ap(\delta)$.
    \label{prop:aii} 
\item Addition \coqfile{absolute_prec.v\#L29}:
    If $k \in \mathbb{R}$, then $a + k \sim \tilde{a} + k \;\apb\; \rp(\alpha)$.
    \label{prop:aiii} 
\item Constant Scaling \coqfile{absolute_prec.v\#L36}:
    If $k \in \mathbb{R}$, then $ka \sim k\tilde{a} \;\apb\; \ap(|k|\alpha)$.
    \label{prop:aiv} 
\item Composition \coqfile{absolute_prec.v\#L44}:
    If $b \sim \tilde{b} \;\apb\;\ap(\beta)$, then $a + b \sim \tilde{a} + \tilde{b} \;\apb\; \ap(\alpha + \beta)$.
    \label{prop:av}
\item Triangle inequality \coqfile{absolute_prec.v\#L52}:
    If $\tilde{a} \sim \tilde{\tilde{a}} \;\apb\; \ap(\delta)$, then $a \sim \tilde{\tilde{a}} \;\apb\; \ap(\alpha + \delta)$.
    \label{prop:avi}
\end{enumerate}

As with absolute error, absolute precision forms a metric on $\R$. From the above properties, the following theorems describe how absolute precision propagates due to multiplication and division.

\begin{theorem}\coqfile{absolute_prec.v\#L80}
Let  $a \sim \tilde{a} \;\apb\; \ap(\alpha),\; b \sim \tilde{b} \;\apb\; \ap(\beta)$. Then
\begin{align*}
	ab \sim \tilde{a} \tilde{b} \;\apb\; \ap(|\tilde{a} \beta| + |\tilde{b} \alpha | + \alpha \beta).
\end{align*}
\label{thm:apmult}
\end{theorem}
\begin{theorem}\coqfile{absolute_prec.v\#L96}
Let  $a \sim \tilde{a} \;\apb\; \ap(\alpha),\; b \sim \tilde{b} \;\apb\; \ap(\beta)$. Suppose additionally that $|\tilde{b}| > \beta$. Then
\begin{align*}
    \frac{a}{b} \sim \frac{\tilde{a}}{\tilde{b}} \;\apb\; \ap\left( \frac{|\tilde{a}| \beta + |\tilde{b}| \alpha}{|\tilde{b}|(|\tilde{b}| - \beta)} \right).
\end{align*}
\label{thm:apdiv}
\end{theorem}

We uncovered another issue in the absolute precision property statements: Olver additionally claims the following equation to be equivalent to \Cref{eq:rp1}:
\begin{equation}\label{eq:rp-ap-conversion-bad}
    \ln a \sim \ln \tilde{a} \;\apb\;\ap(\alpha)
\end{equation}
This is not quite correct for two reasons. Firstly, the equation is ill-defined for negative $a$ and $\tilde{a}$. Secondly, Equation \ref{eq:rp1} uses relative precision which restricts $a$ and $\tilde{a}$ to be non-zero and have the same sign. By contrast, Equation \ref{eq:rp-ap-conversion-bad} uses absolute precision, which is unrestricted with respect to sign. Therefore, we modify the equation to take absolute values before logarithms and only state and prove the forwards direction, namely that Equation \ref{eq:rp1} implies Equation~(\ref{eq:rp-ap-conversion}) below:
\refstepcounter{equation}\label{eq:rp-ap-conversion}
\begin{equation}
    \ln |a| \sim \ln |\tilde{a}| \;\apb\;\ap(\alpha)
    \tag{\theequation ,\coqfilen{relative_prec.v\#L79}}.
    \label{dummy}
\end{equation}

We now proceed to show how this error analysis can be used in practice.

\section{Error Analysis of a Simple Program}\label{sec:example}

With an alternative model for floating-point arithmetic defined in \Cref{ssec:alt_model} and the relative precision of basic arithmetic operations established in \Cref{thm:rp_ops}, we now leverage the metric properties of relative precision to derive a sound upper bound on the total rounding error accumulated over a sequence of floating-point operations. As an example in our mechanization, we proved the following relative precision bound for the inner product of two vectors.

\begin{theorem}\coqfile{relative_prec_ops.v\#L185}\label{thm:example}
Let $x,y \in \mathbb{F}^n$ be vectors of floating-point numbers such that $x_i * y_i > 0$ for every $i \le n$. Let $s = \sum_{i=1}^{n} (x_i * y_i)$ denote the inner product of $x$ and $y$ computed in infinite precision, and let $\tilde{s}$ denote the inner product computed in floating-point arithmetic. Then
\begin{align*}
	s \sim \tilde{s} \;\apb\; \rp\left(\frac{nu}{1-u}\right),
\end{align*}
\label{thm:rpmult}
where $u$ is the unit roundoff.
\end{theorem}

\begin{proof}
The proof follows naturally by induction over $n$, using \Cref{cor:rp_add} and the metric properties of relative precision. These properties simplify the analysis by ensuring that errors propagate in a structured and composable manner. In particular, they allow the error of a sum to be expressed in terms of the errors of its components, making the inductive argument both natural and straightforward.

For the base case, we show that
\begin{equation}
s_1 \sim \tilde{s}_1 \;\apb\; \rp\left(\frac{u}{1-u}\right),
\end{equation}
which follows directly from \Cref{thm:rp_ops}. 

For the inductive step, we assume
$s_{k} \sim \tilde{s}_{k} \;\apb\; \rp\left(\frac{ku}{1-u}\right)$
and aim to show 
\begin{equation}
    s_{k+1} \sim \tilde{s}_{k+1} \;\apb\; \rp\left(\frac{(k+1)u}{1-u}\right),
\end{equation}
where $s_k$ is the $k$th partial sum $s_k = x_1y_1 + \dots + x_ky_k$. 
By definition, we have
\begin{equation}
s_{k+1} = s_k + x_{k+1}y_{k+1}
\end{equation}
and 
\begin{equation}
\tilde{s}_{k+1} = \rho(\tilde{s}_k + \rho(x_{k+1}y_{k+1}))
\end{equation}
where $\rho : \R \rightarrow \F$ is a well-defined rounding function. 

Introducing the intermediate value 
\begin{equation}
\hat{s}_{k+1} = \tilde{s}_k +  \rho(x_{k+1}y_{k+1}),
\end{equation}
we apply \Cref{cor:rp_add}, \Cref{thm:rp_ops}, and the induction hypothesis  to obtain
\begin{equation}
s_{k+1} \sim \hat{s}_{k+1} \;\apb\; \rp\left(\frac{ku}{1-u}\right).
\end{equation}

Furthermore, by \Cref{thm:rp_ops}, we also have
\begin{equation}
\hat{s}_{k+1} \sim \tilde{s}_{k+1} \;\apb\; \rp\left(\frac{u}{1-u}\right).
\end{equation}

The result follows by the triangle inequality.
\end{proof}

The restriction in \Cref{thm:example} that each corresponding product $x_iy_i$ of the 
vectors $x$ and $y$ is strictly positive for all $i \le n$ ensures that all elements 
of the sums $s_i$ and $\tilde{s}_i$ are non-zero and of the same sign, a necessary condition 
for using the relative precision metric as it is defined in \Cref{def:rp}. If the elements 
of the vectors do not satisfy this restriction, absolute precision provides a more appropriate 
measure of the error. 

Since relative error bounds for computing the inner product of two vectors are well established, we can use \Cref{thm:rp_re_conversion} to compare the relative precision bound derived in \Cref{thm:example} with known relative error bounds. For vectors whose components satisfy the restriction that each corresponding product $x_iy_i$ of the 
vectors $x$ and $y$ is strictly positive for all $i \le n$, Higham~\cite{higham02} shows that the relative error of the inner product can be bounded as 
\begin{equation}\label{eq:ip_rel_err_1}
    \left| \frac{s - \tilde{s}}{s} \right| \le \frac{nu}{1-nu}.
\end{equation}
From equation \Cref{thm:rp_re_conversion} we can express the relative error bound in terms of 
relative precision. Substituting the result of \Cref{thm:example} into \Cref{thm:rp_re_conversion}, we obtain
\begin{equation}
        \left| \frac{s - \tilde{s}}{s} \right| \le \frac{nu}{1-(n+1)u},
\end{equation}
\noindent which provides a slightly looser bound than \Cref{eq:ip_rel_err_1} in practice.

Thus, the cost of using relative precision instead of relative error in this setting is a slightly looser bound, but the benefit is that the metric properties of relative precision provide a structured way of composing error bounds across multiple operations. 

\section{Proof Engineering}
Our Rocq mechanization is based on real numbers and uses the rounding modes defined in the Flocq library \cite{boldo11} to reason about the relationships between real numbers and their floating-point approximations; these rounding modes correspond to the rounding functions in \Cref{tab:rnd_modes}. We represent the set of real numbers corresponding to a given floating-point format using the Flocq FLX format, which features unbounded exponents. This choice aligns with the assumptions made in the theorems and definitions from previous sections that no underflows or overflows occur. 

Our definitions of absolute and relative precision are formulated over real numbers as defined in the Mathematical Components library (MathComp \cite{mathcomp22}) for real analysis. We use MathComp and SSReflect for their robust tactics in ring reasoning, number theory, and algebraic rewriting. Our mechanization of \Cref{thm:example} encodes the inner product as a right fold over Rocq lists of real numbers. 

\section{Related Work}

The standard error model is in most automated and manual floating-point error analysis tools, such as VCFloat2 \cite{appel24}, FPTaylor \cite{solovyev18}, PRECiSA, Daisy \cite{Darulova:2018:Daisy}, Colibri2 \cite{colibri21}, Gappa \cite{dedinechin:2011:gappa}. VCFloat2 is tool developed in Rocq for verifying that C programs implement an approximation of a real-number specification. FPTaylor, Daisy, Colibri2, and Gappa provide varying levels of automation for computing floating-point error bounds. The formalization of floating-point arithmetic in Rocq, as well as analysis of various numerical algorithms, has been described by Boldo and Melquiond \cite{boldo17}.

Relative precision has been employed in various analyses, for example, as a loss function in modeling applications---often referred to as the accuracy ratio or quotient error, along with derived functions such as the logarithmic or squared quotient error \cite{tofallis15}---and in accuracy statistics more generally \cite{kitchenham01}, where it helps mitigate certain types of model selection bias. Botchkarev provides an overview of these metrics, particularly in the context of machine learning models \cite{botchkarev19}.

Static analysis tools have also leveraged the metric properties of relative precision. For example, NumFuzz~\cite{kellison24} incorporates floating-point error analysis into its type system using relative precision, enabling precise and scalable reasoning about the numerical accuracy of programs. Bean~\cite{kellison25} similarly employs relative precision for backward error analysis, providing a type-based framework for analyzing numerical stability. Bean extends the definition proposed by Olver to all real numbers as follows:
\begin{equation}\label{eq:bean}
    \textit{RP}(x,y) = \begin{cases}
        |\ln(x/y)| & \text{if } \text{sign}(x) = \text{sign}(y) \text{ and } x, y \ne 0 \\
        0 & \text{if } x = y = 0\\
        \infty & \text{otherwise}
    \end{cases}.
\end{equation}

In numerical analysis, relative precision has also been used to reason about the stability of algorithms under composition. For example, Beltr\'{a}n et al.~\cite{beltran2023condition} adopt the coordinate-free relative error metric introduced by Pryce~\cite{pryce1984relative} in 1984, building on Olver’s original notion of relative precision.

% Compoarison with Flocq: FLT: unbounded exponents, gradual underflow. FLX assumes no underflow and overflow, which is what we're doing. Bounds in other tools: frama-c handles underflow and overflow, but automated solvers have various success (Alt-Ergo, Colibri2, gappa).

% Naive type system is interval analysis

\section{Future Work \& Conclusion}

In this paper we presented a formalization in Rocq of an alternative to the standard model for floating-point error, first written by Olver. By formalizing this theory in Rocq, we discovered undefined scenarios in the definitions related to sign (Property~\ref{prop:iv} and Equation~\ref{eq:rp-ap-conversion-bad}) and an equivalence between Equation~\ref{eq:rp1} and Equation~\ref{eq:rp-ap-conversion-bad} that was false due to a separate issue with signs. By making these assumptions explicit in Rocq, our analysis improves upon the soundness of relative precision.

Mechanizing this error arithmetic provides opportunities for use in other tools. We plan to investigate whether this simplifies static analysis of floating-point programs. The most straightforward future work is building relative precision for more floating-point operations, such as the IEEE 754 supported by Flocq; in this case, while the error analysis would be more complex (often, handling exceptional behaviors is a large portion of the proof engineering effort), it could allow full proofs of correctness for more robust numerical programs. 

% Claim: backwards error analysis is more convenient using relative precision
% Future work: Abstract interpretation with this abstract domain may be better.
Another potential application of relative precision is backwards error analysis, which can be more convenient for detecting sensitivity of input on a program's error; it may be that relative precision simplifies this analysis. 
Backwards error analysis using \Cref{eq:bean} can be mechanized using MathComp's extensive extended real number support ($\mathbb{R} \cup \{\infty\}$).

Similarly, tools such as Gappa or other (relational) abstract interpretation-based tools use heuristics to better estimate ranges of inputs to provide ranges of worst-case error analysis but can suffer from scalability issues as the number of variables and floating-point operations increases; we wish to investigate whether the relative precision-based heuristics (or a combination of standard model and Olver's model) could simplify analysis and hence improve computational efficiency.

In conclusion, this formalization motivates further research in error analysis and provides another tool in the analyst's toolbox to more easily write safe, correct numerical programs.

\section*{Acknowledgments}

This material is based upon work supported by the National Science Foundation Graduate Research Fellowship Program under Grant No. 2139899. Any opinions, findings, and conclusions or recommendations expressed in this material are those of the author(s) and do not necessarily reflect the views of the National Science Foundation.

Sandia National Laboratories is a multimission laboratory managed and operated by National Technology \& Engineering Solutions of Sandia, LLC, a wholly owned subsidiary of Honeywell International Inc., for the U.S. Department of Energy's National Nuclear Security Administration under contract DE-NA0003525.

%\nocite{*}
% Note: using eprint or url in .bib. field is meant for non-paywalled PDFs. Interesting!
\bibliographystyle{eptcs}
\bibliography{generic}
\end{document}